\documentclass[11pt,reqno]{amsart}
\usepackage{amssymb,latexsym}
\usepackage{graphicx}
\usepackage{url}		
\calclayout

\makeatother

\theoremstyle{plain}
\numberwithin{equation}{section}
\newtheorem{thm}{Theorem}[section]
\newtheorem{theorem}[thm]{Theorem}
\newtheorem{lemma}[thm]{Lemma}

\newtheorem{remark}[thm]{Remark}

\newtheorem{corollary}[thm]{Corollary}

\begin{document}
\setcounter{page}{1}

\title[General infinite series evaluations involving Fibonacci numbers]{General infinite series evaluations involving Fibonacci numbers and the Riemann Zeta function}
\author{Robert Frontczak}
\address{Landesbank Baden-W\"urttemberg, 70173 Stuttgart, Germany}
\email{robert.frontczak@lbbw.de}

%
\author{Taras Goy}
\address{Vasyl Stefanyk Precarpathian National University, 76018 Ivano-Frankivsk, Ukraine}
\email{taras.goy@pnu.edu.ua}

\begin{abstract}
The purpose of this article is to present closed forms for various types of infinite series
involving Fibonacci (Lucas) numbers and the Riemann zeta function at integer arguments.
\end{abstract}

\maketitle

\section{Motivation and introduction}

This paper is devoted to combine two very popular and important mathematical objects:
the Riemann zeta function and Fibonacci numbers.
Both objects have been studied intensively and are well understood but identities connecting
them are not documented in the mathematical literature. In this article, we attempt to fill this gap. Using generating functions, we exhibit many interesting infinite series identities derived by fairly routine arguments.

Recall that the Riemann zeta function $\zeta(s), s\in\mathbb{C},$ is defined by \cite{Abramowitz}
\begin{equation*}
\zeta(s) = \sum_{k=1}^\infty \frac{1}{k^s}, \quad \Re(s)>1.
\end{equation*}
The analytical continuation to all $s\in\mathbb{C}$ with $\Re(s)>0,s\neq 1,$ is given by
\begin{equation*}
\zeta(s) = (1-2^{1-s})^{-1} \sum_{k=1}^\infty \frac{(-1)^{k+1}}{k^s}.
\end{equation*}
The evaluation of $\zeta(s)$ at integer arguments is an old problem that still challenges the mathematical community. For even positive integer arguments this problem was completely solved by Euler showing that
\begin{equation*}
\zeta(2n) = (-1)^{n+1}\frac{(2\pi)^{2n}}{2(2n)!} B_{2n},
\end{equation*}
where $B_n$ are the Bernoulli numbers \cite{Arakawa}. For odd integer arguments, the problem is still open. Much more information about $\zeta(s)$ is contained in the textbooks \cite{Edwards,Srivastava}, among others. 

On the other hand, Fibonacci numbers are one of the most famous integer sequences in the mathematical world.
The Fibonacci numbers $F_n$ and the companion sequence of Lucas numbers $L_n$ are defined for $n\geq 0$
as $F_{n+2} = F_{n+1} + F_n$ and $L_{n+2} = L_{n+1} + L_n$ with initial conditions $F_0 = 0, F_1 = 1$, $L_0=2$ and $L_1=1$,
respectively. The Binet formulas are given by
\begin{equation*}
F_n = \frac{\alpha^n-\beta^n}{\alpha-\beta}, \qquad L_n = \alpha^n + \beta^n,
\end{equation*}
where $\alpha$ is the golden ratio, i.e., $\alpha = \frac{1+\sqrt{5}}{2}$ and $\beta=-1/\alpha=\frac{1-\sqrt{5}}{2}$.
The sequences $(F_n)_{n\geq 0}$ and $(L_n)_{n\geq 0}$ possess many interesting properties and appear in mathematical branches such as combinatorics and graph theory. See \cite{Koshy} for more details. They are indexed in the On-Line Encyclopedia of Integer Sequences \cite{OEIS} with entries A000045 and A000032, respectively.

Infinite series evaluations involving Fibonacci (Lucas) numbers and the zeta function are rare. Very recently, the following evaluations involving the Riemann zeta function at positive even integer argument and scaled even Fibonacci (Lucas) numbers were stated in \cite{Frontczak-Elem}:
\begin{equation*}
\sum_{k=1}^\infty \zeta(2k) \frac{F_{2k}}{5^k} = \frac{\pi}{2\sqrt{5}}\tan\Big ( \frac{\pi}{2\sqrt{5}}\Big ),
\end{equation*}
\begin{equation*}
\sum_{k=1}^\infty \zeta(2k) \frac{L_{2k}}{5^k} = \frac{\pi}{2\sqrt{5}}\tan\Big ( \frac{\pi}{2\sqrt{5}}\Big ) + 1.
\end{equation*}

A still more appealing identity involving $\zeta(s)$ at odd integer argument and Fibonacci numbers comes as another problem proposal from \cite{Frontczak-Adv}:
\begin{equation}\label{id1}
\sum_{k=1}^\infty \zeta(2k+1) \frac{F_{2k}}{5^k} = \frac{1}{2}.
\end{equation}
Interestingly, the Lucas counterpart does not possess such a nice structure:
\begin{equation*}
\sum_{k=1}^\infty \zeta(2k+1) \frac{L_{2k}}{5^k} = \frac{3}{2} - 2\sum_{n=1}^\infty \frac{1}{n(5n^2-5n+1)(5n^2+5n+1)}.
\end{equation*}
A few more such relations can be found in \cite{Frontczak-NNTDM}.

The goal of this article is to continue the research in this direction and to present
more closed forms for some types of infinite series involving Fibonacci (Lucas) numbers and the Riemann zeta function. To prove the results, we will mainly work with generating functions and some series evaluations. In addition, in some of our proofs we will apply properties of the digamma function $\psi(z), z\in\mathbb{C}$. Recall that  $\psi(z)$ is the first logarithmic derivative of the gamma function, i.e.,
\begin{equation*}
\psi(z) = (\ln \Gamma(z))' = \frac{\Gamma'(z)}{\Gamma(z)},
\end{equation*}
where $\Gamma(z)$ is the gamma function \cite{Abramowitz}. 

The digamma function possesses the following properties:
\begin{equation}\label{psi_id1}
\psi(z+1) = \psi(z) + \frac{1}{z},
\end{equation}
\begin{equation}\label{psi_id2}
\psi(z+1) = -\gamma + \sum_{n=1}^\infty \Big ( \frac{1}{n} - \frac{1}{n+z}\Big ), \qquad z\neq -1, -2, \ldots,
\end{equation}
and the reflection property
\begin{equation}\label{psi_id3}
\psi(1-z) - \psi(z) = \pi \cot\pi z,
\end{equation}
where $\gamma$ is famous the Euler-Mascheroni constant
\begin{equation*}
\gamma = \lim_{n\rightarrow \infty}\Big(\sum_{k=1}^n\frac{1}{k} - \ln n\Big) = 0,5772156649\ldots
\end{equation*}

These properties will be employed in some of the proofs below.

\section{Main results}

The following lemma will be used repeatedly in this section.
\begin{lemma}\cite{Frontczak-NNTDM} \label{mainlem1}
	The following identity holds true:
	\begin{equation}
	\sum_{n=1}^\infty \frac{1}{n^2+n-1} = 1+\frac{\sqrt{5}}{5}\pi\tan\!\frac{\sqrt{5} \pi}{2}.
	\end{equation}
\end{lemma}
\begin{theorem}\label{thm1}
	For $m\geq0$, we have
	\begin{equation} \label{main1}
	\sum_{k=1}^\infty (\zeta(2k)-1) F_{2k+m-1} = \frac{\pi}{2\sqrt{5}}\tan\!\frac{\sqrt{5}\pi}{2}L_{m} + \frac{1}{2}F_{m+2}
	\end{equation}
	and
	\begin{equation} \label{main2}
	\sum_{k=1}^\infty (\zeta(2k)-1) L_{2k+m-1} = \frac{\sqrt{5}\pi}{2}\tan\frac{\sqrt{5}\pi}{2}F_{m} + \frac{1}{2}L_{m+2}.
	\end{equation}
\end{theorem}
\begin{proof}
	From \cite[p.~281]{Srivastava} we know that
	\begin{equation*}
	\sum_{k=1}^\infty (\zeta(2k)-1) z^{2k-1} = -\frac{\pi}{2} \cot\pi z + \frac{3z^2-1}{2z(z^2-1)}, \quad |z|<2.
	\end{equation*}
	This gives with $z=\alpha$
	\begin{equation*}
	\sum_{k=1}^\infty (\zeta(2k)-1) \alpha^{2k-1} = -\frac{\pi}{2} \cot\pi\alpha + \frac{\alpha^2}{2},
	\end{equation*}
	where we have used that $3\alpha+2=\alpha^4$ and $\alpha^2=\alpha+1$. Hence,
	\begin{equation*}
	\sum_{k=1}^\infty (\zeta(2k)-1) \alpha^{2k+m-1} = -\frac{\pi}{2} \alpha^{m} \cot\pi \alpha + \frac{\alpha^{m+2}}{2}.
	\end{equation*}
	
	In the same way, we get
	\begin{equation*}
	\sum_{k=1}^\infty (\zeta(2k)-1) \beta^{2k+m-1} = -\frac{\pi}{2} \beta^{m} \cot\pi \beta + \frac{\beta^{m+2}}{2}.
	\end{equation*}
	Combining these equations according to the Binet formula and making use of the fact that $\cot(\pi/2-x)=\tan x$
	we obtain
	\begin{align*}
	\sum_{k=1}^\infty (\zeta(2k)-1) F_{2k+m-1} &= \frac{\pi}{2 \sqrt{5}}\big(\beta^{m} \cot\pi\beta- \alpha^{m} \cot\pi\alpha\big) + \frac{1}{2}F_{m+2} \\
	& = \frac{\pi}{2 \sqrt{5}}\Big (\beta^{m} \tan\frac{\sqrt{5}\pi}{2}+ \alpha^{m} \tan\frac{\sqrt{5}\pi}{2}\Big) + \frac{1}{2}F_{m+2}
	\end{align*}
	and the proof of \eqref{main1} is completed. The identity \eqref{main2} is proved similarly and omitted.
\end{proof}

Explicit examples for $m=0$ and $m=1$ are
\begin{equation*}
\sum_{k=1}^\infty (\zeta(2k)-1) F_{2k-1} = \frac{\pi}{\sqrt{5}}\tan\frac{\sqrt{5}\pi}{2} + \frac{1}{2},
\end{equation*}
\begin{equation*}
\sum_{k=1}^\infty (\zeta(2k)-1) L_{2k-1} = \frac{3}{2},
\end{equation*}
\begin{equation*}
\sum_{k=1}^\infty (\zeta(2k)-1) F_{2k} = \frac{\pi}{2\sqrt{5}}\tan\frac{\sqrt{5}\pi}{2}+ 1,
\end{equation*}
and
\begin{equation*}
\sum_{k=1}^\infty (\zeta(2k)-1) L_{2k} = \frac{\sqrt{5}\pi}{2}\tan\frac{\sqrt{5}\pi}{2} + 2.
\end{equation*}
\begin{remark}\label{mainrem}
	 We point out that instead of proving \eqref{main2} directly (as we did implicitly), it can also be deduced from \eqref{main1} using $5F_n = L_{n+1} + L_{n-1}$ and $L_n = F_{n+1} + F_{n-1}$.
\end{remark}
	
A proof comparable to the one given for Theorem \ref{thm1} yields the following series. 
The similar results are deducible from the other theorems.
\begin{corollary}
	For $n\geq1$, 
	\begin{equation*}
	\sum_{k=1}^{\infty}
	\frac{\zeta(2k)-1}{n^{2k-1}}F_{2k-1}=\frac{1}{\sqrt5}\cdot\frac{\pi\sin\frac{\pi\sqrt5}{n}}{\cos\frac{\pi\sqrt5}{n}-\cos\frac{\pi}{n}}+\frac{2n(n^4-5n^2+3)}{(2n^2-3+\sqrt5)(2n^2-3-\sqrt5)}\end{equation*}
	and
	\begin{equation*}
	\sum_{k=1}^{\infty}
	\frac{\zeta(2k)-1}{n^{2k-1}}L_{2k-1}=\frac{\pi\sin\frac{\pi}{n}}{\cos\frac{\pi}{n}-\cos\frac{\pi\sqrt5}{n}}-\frac{2n(n^4-n^2+3)}{(2n^2-3+\sqrt5)(2n^2-3-\sqrt5)}.
	\end{equation*}
\end{corollary}

As could be expected, the formula including the odd zeta values is more involved and possesses a semi-closed form.
\begin{theorem}\label{thm2}
	For $m\geq 0$, we have
	\begin{align} \label{main3}
	\sum_{k=1}^\infty (\zeta(2k+1)-1) F_{2k+m}&= \frac{F_{m}+F_{m+2}}{2}\nonumber\\
	- F_{m}&\frac{\pi}{\sqrt{5}}\tan\frac{\sqrt{5}\pi}{2}
	-\frac{F_{m}}{2} \sum_{n=1}^\infty \frac{1}{n(n+1)(n^2+3n+1)},
	\end{align}
	\begin{align} \label{main4}
	\sum_{k=1}^\infty (\zeta(2k+1)-1) L_{2k+m}& = \frac{L_{m}+L_{m+2}}{2}\nonumber\\
	- L_{m}&\frac{\pi}{\sqrt{5}}\tan\frac{\sqrt{5}\pi}{2}
	-\frac{L_{m}}{2} \sum_{n=1}^\infty \frac{1}{n(n+1)(n^2+3n+1)}.
	\end{align}
\end{theorem}
\begin{proof}
	From \cite[p.~280]{Srivastava} we have the following generating function:
	\begin{equation*}
	\sum_{k=1}^\infty (\zeta(2k+1)-1) z^{2k} = (1-\gamma) - \frac{1}{2}\big (\psi(2+z) + \psi(2-z)\big), \quad |z|<2.
	\end{equation*}
	This relation combined with the Binet formula and $2-\alpha=\alpha^{-2}=\beta^2$, $2-\beta=\alpha^2$ yields
	\begin{align*}
	\sum_{k=1}^\infty (\zeta(2k+1)-1) F_{2k+m} & = (1-\gamma)F_m + \frac{1}{2\sqrt{5}}\big (\beta^m \psi(\beta^2+1)-\alpha^m \psi(\beta^2)\big ) \\
	& - \frac{1}{2\sqrt{5}}\big (\alpha^m \psi(\alpha^2+1) -\beta^m \psi(\alpha^2)\big ).
	\end{align*}
	Now, we note that
	\begin{align*}
	\beta^m \psi(\beta^2+1)-\alpha^m \psi(\beta^2)	= L_m \big(\psi(\beta^2+1)-\psi(\beta^2)\big) - \alpha^m \psi(\beta^2+1) + \beta^m \psi(\beta^2)
	\end{align*}
	and
	\begin{align*}
	\alpha^m \psi(\alpha^2+1)-\beta^m \psi(\alpha^2)	= L_m \big(\psi(\alpha^2+1)-\psi(\alpha^2)\big) - \beta^m \psi(\alpha^2+1) + \alpha^m \psi(\alpha^2).
	\end{align*}
	Gathering terms and keeping in mind property \eqref{psi_id1} we arrive at
	\begin{equation*}
	\sum_{k=1}^\infty (\zeta(2k+1)-1) F_{2k+m} = (1-\gamma)F_m + \frac{1}{2}L_m + \frac{1}{2\sqrt{5}}(S_1 + S_2),
	\end{equation*}
	with
	\begin{gather*}
	S_1 = \beta^m \psi(\beta^2) - \alpha^m \psi(\alpha^2),\qquad
	S_2 = \beta^m \psi(\alpha^2+1) - \alpha^m \psi(\beta^2+1).
	\end{gather*}
	
	Next, we apply property \eqref{psi_id2} to get
	\begin{align*}
	S_1 & = \gamma\sqrt{5} F_m + \sum_{n=1}^\infty \Big (\frac{\beta^m}{n}-\frac{\beta^m}{n+\beta}-\frac{\alpha^m}{n}+\frac{\alpha^m}{n+\alpha}\Big ) \\
	& = \gamma\sqrt{5} F_m - \sum_{n=1}^\infty \frac{(n-1)\sqrt{5}F_m + n\sqrt{5}F_{m-1}}{n(n+\alpha)(n+\beta)} \\
	& = \gamma\sqrt{5} F_m - \sqrt{5}F_{m+1}\sum_{n=1}^\infty \frac{1}{n^2+n-1}+\sqrt{5}F_{m}\sum_{n=1}^\infty \frac{1}{n(n^2+n-1)}.
	\end{align*}
	
	Similarly for $S_2$,
	\begin{align*}
	S_2 & =  \gamma\sqrt{5} F_m + \sum_{n=1}^\infty \frac{-(3n+1)\sqrt{5}F_m + n\sqrt{5}F_{m+2}}{n(n+\alpha^2)(n+\beta^2)} \\
	& =  \gamma\sqrt{5} F_m + \sqrt{5}F_{m+1}\sum_{n=1}^\infty \frac{1}{n^2+3n+1} \\
	& - 2\sqrt{5}F_{m}\sum_{n=1}^\infty \frac{1}{n^2+3n+1} - \sqrt{5}F_{m}\sum_{n=1}^\infty \frac{1}{n(n^2+3n+1)} .
	\end{align*}
	Hence,
	\begin{align*}
	\sum_{k=1}^\infty (\zeta(2k+1)-1) F_{2k+m}& = F_m + \frac{L_m}{2}- \frac{F_{m+1}}{2}\sum_{n=1}^\infty \frac{1}{n^2+n-1} +
	\frac{F_{m}}{2}\sum_{n=1}^\infty \frac{1}{n(n^2+n-1)} \\
	&+ \Big (\frac{F_{m+1}}{2} - F_{m}\Big )\sum_{n=1}^\infty \frac{1}{n^2+3n+1} - \frac{F_{m}}{2}\sum_{n=1}^\infty \frac{1}{n(n^2+3n+1)}.
	\end{align*}
	
	To simplify further, we note that
	\begin{equation*}
	\sum_{n=1}^\infty \frac{1}{n^2+n-1} = \sum_{n=0}^\infty \frac{1}{(n+1)^2+(n+1)-1} = 1 + \sum_{n=1}^\infty \frac{1}{n^2+3n+1},
	\end{equation*}
	so that from Lemma \ref{mainlem1} we get
	\begin{equation*}
	\sum_{n=1}^\infty \frac{1}{n^2+3n+1} = \frac{\pi}{\sqrt{5}} \tan\frac{\sqrt{5}\pi}{2}.
	\end{equation*}
	Also
	\begin{equation*}
	\sum_{n=1}^\infty \Big (\frac{1}{n(n^2+3n+1)}-\frac{1}{n(n^2+3n+1)}\Big) = 1 - \sum_{n=1}^\infty \frac{1}{n(n+1)(n^2+3n+1)}
	\end{equation*}
	and we finally end with
	\begin{align*}
	\sum_{k=1}^\infty (\zeta(2k+1)-1) F_{2k+m} & = \frac{3F_{m}+L_{m}-F_{m+1}}{2} -
	F_{m}\frac{\pi}{\sqrt{5}}\tan\frac{\sqrt{5}\pi}{2}\\
	&-\frac{F_{m}}{2} \sum_{n=1}^\infty \frac{1}{n(n+1)(n^2+3n+1)},
	\end{align*}
	from which our statement follows upon simple manipulations of the first term on the right-hand side. This completes the proof of \eqref{main3}. The statement \eqref{main4} can be proved either analogously or using the relations from Remark \ref{mainrem}.
\end{proof}

When $m=0$, then from \eqref{main3} we get the expressions
\begin{gather*}
\sum_{k=1}^\infty (\zeta(2k+1)-1) F_{2k} = \frac{1}{2},\\
\sum_{k=1}^\infty (\zeta(2k+1)-1) L_{2k}=\frac{5}{2}-\frac{2\pi}{\sqrt{5}}\tan\frac{\sqrt{5}\pi}{2}
-\sum_{n=1}^\infty \frac{1}{n(n+1)(n^2+3n+1)}.
\end{gather*}

In view of \eqref{id1} we arrive at the beautiful result
\begin{displaymath}
\sum_{k=1}^\infty \zeta(2k+1) \frac{F_{2k}}{5^k} = \sum_{k=1}^\infty (\zeta(2k+1)-1) F_{2k} = \frac{1}{2}.
\end{displaymath}

The next result generalizes an identity from \cite{Frontczak-NNTDM}.
\begin{theorem}\label{thm3}
	For $m\geq0$, we have
	\begin{gather*} 
	\sum_{k=2}^\infty (\zeta(k)-1) F_{k+m-1}	= F_{m+1} + F_{m-1}\frac{\pi}{\sqrt{5}}\tan\frac{\sqrt{5}\pi}{2}
	+ F_{m} \sum_{n=1}^\infty \frac{1}{n(n^2+n-1)},\\
	\sum_{k=2}^\infty (\zeta(k)-1) L_{k+m-1}= L_{m+1} + L_{m-1}\frac{\pi}{\sqrt{5}}\tan\frac{\sqrt{5}\pi}{2}
	+ L_{m} \sum_{n=1}^\infty \frac{1}{n(n^2+n-1)}.
	\end{gather*}
\end{theorem}
\begin{proof}
	Here, we work with the generating function
	\begin{equation*}
	\sum_{k=2}^\infty (\zeta(k)-1) z^{k-1} = 1-\gamma - \psi(2-z), \quad |z|<2,
	\end{equation*}
	which also comes from \cite[p.~280]{Srivastava}. This gives
	\begin{equation*}
	\sum_{k=2}^\infty (\zeta(k)-1) F_{k+m-1}
	= (1-\gamma)F_{m} + \frac{1}{\sqrt{5}}\big (\beta^{m} \psi(\alpha+1)
	-\alpha^{m} \psi(\beta+1)\big ).
	\end{equation*}
	
	The remainder of the proof is as above and we leave it as an exercise.
\end{proof}

When $m=0$, we use the Fibonacci relation $F_{-n}=(-1)^{n+1}F_{n}$ to see that
\begin{equation*}
\sum_{k=2}^\infty (\zeta(k)-1) F_{k-1} = 1 + \frac{\pi}{\sqrt{5}}\tan\frac{\sqrt{5}\pi}{2},
\end{equation*}
which appears in \cite{Frontczak-NNTDM}.

\section{Further related series}

In this section we study some series that are closely related to the series form the last section.
\begin{theorem}\label{thm4}
	Let $m\geq 0$. Then we have
	\begin{equation*} 
	\sum_{k=1}^\infty (\zeta(2k)-1) \frac{F_{2k+m-1}}{k} = F_{m-1}\ln\Big (\!-\pi\sec\frac{\sqrt{5}\pi}{2}\Big)
	+ \frac{2}{\sqrt{5}}L_{m-1}\ln\alpha
	\end{equation*}
	and
	\begin{equation*} 
	\sum_{k=1}^\infty (\zeta(2k)-1) \frac{L_{2k+m-1}}{k} = L_{m-1}\ln\Big (\!-\pi\sec\frac{\sqrt{5}\pi}{2}\Big)
	+ 2 \sqrt{5} F_{m-1}\ln\alpha.
	\end{equation*}
\end{theorem}
\begin{proof}
	We work with the following generating function from \cite[p.~281]{Srivastava}
	\begin{equation*}
	\sum_{k=1}^\infty (\zeta(2k)-1) \frac{z^{2k}}{k} = \ln\big (\pi z(1-z^2)\csc\pi z\big ), \qquad |z|<2.
	\end{equation*}
	This relation yields straightforwardly for $m\geq0$
	\begin{equation*}
	\sum_{k=1}^\infty (\zeta(2k)-1) \frac{F_{2k+m-1}}{k} = \frac{\alpha^{m-1} \ln\big(\!-\pi\alpha^2\csc\pi\alpha\big)
		- \beta^{m-1} \ln\big (\!-\pi\beta^2\csc\pi\beta\big)}{\sqrt{5}}.
	\end{equation*}
	The first expression is obtained by simplification using 
	$\sin\pi \alpha = \sin\pi \beta = \cos\frac{\sqrt{5}\pi}{2}$.
	The proof of the second expression is similar.
\end{proof}
When $m=0$ and $m=1$, then with the use of $L_{-n}=(-1)^n L_n$
\begin{equation*}
\sum_{k=1}^\infty (\zeta(2k)-1) \frac{F_{2k-1}}{k} = \ln\Big (\!-\pi\sec\frac{\sqrt{5}\pi}{2}\Big) - \frac{2}{\sqrt{5}}\ln\alpha,
\end{equation*}
\begin{equation*}
\sum_{k=1}^\infty (\zeta(2k)-1) \frac{L_{2k-1}}{k} = -\ln\Big (\!-\pi\sec\frac{\sqrt{5}\pi}{2}\Big ) + 2\sqrt{5}\ln\alpha,
\end{equation*}
\begin{equation*}
\sum_{k=1}^\infty (\zeta(2k)-1) \frac{F_{2k}}{k} = \frac{4}{\sqrt{5}}\ln\alpha,
\end{equation*}
and
\begin{equation*}
\sum_{k=1}^\infty (\zeta(2k)-1) \frac{L_{2k}}{k} = 2\ln\Big (\!-\pi\csc\frac{\sqrt{5}\pi}{2}\Big ).
\end{equation*}
\begin{theorem}\label{thm5}
	Let $m\geq0$. Then we have
	\begin{equation*} 
	\sum_{k=1}^\infty (\zeta(2k+1)-1) \frac{F_{2k+m}}{2k+1}= (1-\gamma)F_{m} + \frac{L_{m-1}}{2\sqrt{5}}\Big (
		\ln\Big (\!-\pi\sec\frac{\sqrt{5}\pi}{2}\Big) - 2\ln\Gamma(\alpha) - 4\ln\alpha\Big ),
		\end{equation*}
	and
		\begin{equation*}
		\sum_{k=1}^\infty (\zeta(2k+1)-1) \frac{L_{2k+m}}{2k+1}= (1-\gamma)L_{m} + \frac{\sqrt{5} F_{m-1}}{2}\Big (
	\ln\Big (\!-\pi\sec\frac{\sqrt{5}\pi}{2}\Big ) - 2\ln\Gamma(\alpha) - 4\ln\alpha\Big),
	\end{equation*}
	where  $\gamma$ is the Euler-Mascheroni constant.
\end{theorem}
\begin{proof}
	The following generating function is stated in \cite[p.~280]{Srivastava}
	\begin{equation*}
	\sum_{k=1}^\infty (\zeta(2k+1)-1) \frac{z^{2k+1}}{2k+1} = (1-\gamma)z + \frac{1}{2}\ln\frac{\Gamma(2-z)}{\Gamma(2+z)}, \qquad |z|<2.
	\end{equation*}
	This relation yields straightforwardly for $m\geq0$
	\begin{align*}
	\sum_{k=1}^\infty (\zeta(2k+1)-1) \frac{F_{2k+m}}{2k+1}
	= (1-\gamma)F_{m} + \frac{1}{2\sqrt{5}}\Big (
	\alpha^{m-1} \ln\frac{\Gamma(\beta+1)}{\Gamma(\alpha^2+1)}- \beta^{m-1} \ln\frac{\Gamma(\alpha+1)}{\Gamma(\beta^2+1)}\Big ).
	\end{align*}
	The fundamental functional equation of the gamma function, $\Gamma(z+1)=z\Gamma(z)$, yields
	\begin{align*}
	\sum_{k=1}^\infty (\zeta(2k+1)-1) \frac{F_{2k+m}}{2k+1}= (1-\gamma)F_{m} + \frac{1}{2\sqrt{5}}\Big (
	\alpha^{m-1} \ln\frac{\beta \Gamma(\beta)}{\alpha^3 \Gamma(\alpha)}- &\beta^{m-1} \ln\frac{\alpha \Gamma(\alpha)}{\beta^3
		\Gamma(\beta)}\Big ).
	\end{align*}
	Since
	$\frac{\beta \Gamma(\beta)}{\alpha^3 \Gamma(\alpha)} = \Big (\frac{\alpha \Gamma(\alpha)}{\beta^3 \Gamma(\beta)}\Big)^{-1}$, 	we can write the last equation as follows
	\begin{equation*}
	\sum_{k=1}^\infty (\zeta(2k+1)-1) \frac{F_{2k+m}}{2k+1} = (1-\gamma)F_{m} + \frac{L_{m-1}}{2\sqrt{5}}\ln\Big
	(\frac{-\Gamma(\beta)}{\alpha^4 \Gamma(\alpha)}\Big ).
	\end{equation*}
	
	Finally, note that 
	\begin{equation*}
	\frac{\Gamma(\beta)}{\Gamma(\alpha)} = \frac{\pi}{\Gamma^2(\alpha) \sin\pi\alpha},
	\end{equation*}
	where we have used
	$\Gamma(z) \Gamma(1-z) = \frac{\pi}{\sin\pi z}$.
	This completes the first proof. The other one is omitted.
\end{proof}

The special evaluations for $m=0$ and $m=1$ are
\begin{equation*}
\sum_{k=1}^\infty (\zeta(2k+1)-1) \frac{F_{2k}}{2k+1} = \frac{1}{2\sqrt{5}}\Big (
-\ln\Big (\!-\pi\sec\frac{\sqrt{5}\pi}{2}\Big ) + 2\ln\Gamma(\alpha) + 4\ln\alpha\Big ),
\end{equation*}
\begin{equation*}
\sum_{k=1}^\infty (\zeta(2k+1)-1) \frac{L_{2k}}{2k+1} = 2(1-\gamma) + \frac{\sqrt{5}}{2}\Big (
\ln\Big (\!-\sec\frac{\sqrt{5}\pi}{2}\Big ) - 2\ln\Gamma(\alpha) - 4\ln\alpha\Big ),
\end{equation*}
\begin{equation*}
\sum_{k=1}^\infty (\zeta(2k+1)-1) \frac{F_{2k+1}}{2k+1} = 1-\gamma + \frac{1}{\sqrt{5}}\Big (
\ln\Big (\!-\sec\frac{\sqrt{5}\pi}{2}\Big ) - 2\ln\Gamma(\alpha) - 4\ln\alpha\Big ),
\end{equation*}
and the interesting identity
\begin{equation*}
\sum_{k=1}^\infty (\zeta(2k+1)-1) \frac{L_{2k+1}}{2k+1} = 1-\gamma.
\end{equation*}
\begin{theorem}\label{thm6}
	For $m\geq0$, we have
	\begin{align*} 
	\sum_{k=2}^\infty (\zeta(k)-1) \frac{F_{k+m-1}}{k} 
	\!=\! (1-\gamma)F_{m} - \frac{\ln\Gamma(\alpha)}{\sqrt{5}} L_{m-1} + F_{m-1}\ln\alpha 
	+ \frac{\alpha^{m-1}}{\sqrt{5}} \ln\Big(\!-\frac{\pi}{\alpha^2} \sec\frac{\sqrt{5}\pi}{2}\Big)
	\end{align*}
	and
	\begin{align*} 
	\sum_{k=2}^\infty (\zeta(k)-1) \frac{L_{k+m-1}}{k}\!=\!
	 (1-\gamma)L_{m} - \sqrt{5} \ln\Gamma(\alpha) F_{m-1} + L_{m-1}\ln\alpha 
	+ \alpha^{m-1} \ln\Big(\!-\!\frac{\pi}{\alpha^2} \sec\frac{\sqrt{5}\pi}{2}\Big)\!.
	\end{align*}
\end{theorem}
\begin{proof}
	We can prove the statements using \cite[p.~280]{Srivastava}
	\begin{equation*}
	\sum_{k=2}^\infty (\zeta(k)-1) \frac{z^{k}}{k} = (1-\gamma)z + \ln \Gamma(2-z), \qquad |z|<2.
	\end{equation*}
	
	In case of Fibonacci numbers, we can derive
	\begin{equation*}
	\sum_{k=2}^\infty (\zeta(k)-1) \frac{F_{k+m}}{k} = (1-\gamma)F_{k+m} + \frac{1}{\sqrt{5}}\big (\alpha^m \ln (\beta \Gamma(\beta))
	-\beta^m \ln(\alpha\Gamma(\alpha))\big),
	\end{equation*}
	from which the statement is obtained by simplification. The Lucas series is obtained analogously.
\end{proof}

When $m=1$, then the special cases are
\begin{equation*}
\sum_{k=2}^\infty (\zeta(k)-1) \frac{F_{k}}{k} = 1-\gamma - \frac{2}{\sqrt{5}} \bigl(\Gamma(\alpha)+\ln\alpha\bigr)+\frac{1}{\sqrt{5}}\ln\Big (\!-\pi\sec\frac{\sqrt5\pi}{2}\Big),
\end{equation*}
\begin{equation*}
\sum_{k=2}^\infty (\zeta(k)-1) \frac{L_{k}}{k} = 1-\gamma + \ln\Big (\!-\pi\sec\frac{\sqrt5\pi}{2}\Big).
\end{equation*}

\section{Concluding remarks}

In this article we present new closed forms for some types of infinite series involving Fibonacci and Lucas numbers with the Riemann zeta function of integer arguments. To prove our results, we use Binet's formulas, generating functions and some known series evaluations.
Our next work is to establish series evaluations with other popular number sequences, such as Pell, Pell-Lucas, Jacobsthal, Jacobsthal-Lucas, Mersenne and balancing numbers.

\medskip

\noindent MSC2010: 11B39, 11M06, 40C15

\noindent Keywords: Fibonacci number, Lucas number, Riemann zeta function, digamma function.

\end{document}